\documentclass[11pt, final]{article}
\usepackage{a4}
\usepackage{amsmath}%
\usepackage{amstext}%
\usepackage{amssymb}%
\usepackage{showkeys}%
\usepackage{epsfig}%
\usepackage{cite}

\usepackage{listings}

\usepackage[left=3.1cm,right=3.1cm,top=3.3cm,bottom=3.3cm]{geometry}
\usepackage[usenames,dvipsnames]{xcolor}
\usepackage[plainpages=false,pdfpagelabels,colorlinks=true,linkcolor=blue,citecolor=Magenta]{hyperref} 

\usepackage{showkeys}
\usepackage[T1]{fontenc}        
\usepackage{lmodern}            
\usepackage{graphicx}
\usepackage{subfig}		
\usepackage{booktabs}		
\usepackage{multirow}
\usepackage{enumitem}		
\usepackage{listings} 	
\usepackage{mathtools}  
\usepackage{floatrow}		

\usepackage{colortbl}
\definecolor{dunkelgrau}{rgb}{0.8,0.8,0.8}
\definecolor{hellgrau}{rgb}{0.9,0.9,0.9}
\newtheorem{theorem}{Theorem}

\newtheorem{algorithm}[theorem]{Algorithm}
\newtheorem{axiom}{Axiom}

\newtheorem{definition}[axiom]{Definition}

\newtheorem{lemma}[theorem]{Lemma}

\newtheorem{problem}[theorem]{Problem}

\newtheorem{rem}[theorem]{Remark}
\newenvironment{remark}{\rem\rm}{\endrem}

\newcounter{unnumber}

\newtheorem{prf}[unnumber]{Proof.}
\newenvironment{proof}{\prf\rm}{\hfill{$\blacksquare$}\endprf}
\newcommand{\R}{\mathbb{R}}%
\newcommand{\e}{\varepsilon}%
\newcommand{\ol}{\overline}%

\renewcommand{\>}{\right\rangle}
\newcommand{\<}{\left\langle}
\DeclareMathOperator*\inte{int}%
\DeclareMathOperator*\dom{dom}%
\DeclareMathOperator*\B{\overline{\R}}%
\DeclareMathOperator*\argmin{argmin}

\DeclareMathOperator*\dist{dist}
\DeclareMathOperator*\crit{crit}

\title{Proximal-gradient algorithms for fractional programming}

\author{Radu Ioan Bo\c{t} \thanks{University of Vienna, Faculty of Mathematics, Oskar-Morgenstern-Platz 1, A-1090 Vienna, Austria,
email: radu.bot@univie.ac.at. } \and
Ern\"{o} Robert Csetnek \thanks {University of Vienna, Faculty of Mathematics, Oskar-Morgenstern-Platz 1, A-1090 Vienna, Austria,
email: ernoe.robert.csetnek@univie.ac.at. Research supported by FWF (Austrian Science Fund), Lise Meitner Programme, project M 1682-N25.}}

\begin{document}
\maketitle

\noindent \textbf{Abstract.} In this paper we propose two proximal gradient algorithms for fractional programming problems in real Hilbert spaces, where 
the numerator is a proper, convex and lower semicontinuous function and the denominator is a smooth function, either concave or convex. 
In the iterative schemes, we perform a proximal step with respect to the nonsmooth numerator and a gradient step with respect to the smooth 
denominator. The algorithm  in case of a concave denominator has the particularity that it generates sequences which approach both the (global) optimal solutions set and the optimal objective value of the underlying fractional programming problem. In case 
of a convex denominator the numerical scheme approaches the set of critical points of the objective function, provided the latter satisfies 
the Kurdyka-\L{}ojasiewicz property. 
\vspace{1ex}

\noindent \textbf{Key Words.} fractional programming, forward-backward algorithm, convergence rate, convex subdifferential, 
limiting subdifferential, Kurdyka-\L{}ojasiewicz property\vspace{1ex}

\noindent \textbf{AMS subject classification.} 65K05, 90C25, 90C32

\section{Introduction and preliminaries}

Consider the fractional programming problem 
\begin{equation}\label{fr-intr} \ol\theta:=\inf_{x\in S}\frac{f(x)}{g(x)}, 
\end{equation}
where $S$ is a nonempty subset of a real Hilbert space ${\cal H}$, the function $f$ is nonnegative and the function $g$ is positive on $S$. One of the classical methods to handle \eqref{fr-intr} 
is Dinkelbach's procedure (see \cite{dinkelbach, cr}) which relates it to the following optimization problem 
\begin{equation}\label{dink} \inf_{x\in S}\{f(x)-\ol\theta g(x)\}.\end{equation}
If \eqref{fr-intr} has an optimal solution $\bar x \in S$, then this is also an optimal solution to \eqref{dink} and the optimal objective value of the latter is equal to zero. Vice-versa, if \eqref{dink}
has $\bar x \in S$ as an optimal solution and its optimal objective value is equal to zero, then $\bar x$ is an optimal solution to \eqref{fr-intr}, too. This shows that finding an optimal solution to \eqref{fr-intr} can be 
approached by algorithms which solve \eqref{dink}. However, one drawback of this procedure is that this can be done 
in the very restrictive case when the optimal objective value of \eqref{fr-intr} is known.   

One can find in the literature (see \cite{dinkelbach, cr, schaible, ibaraki1981, ibaraki1983}) an iterative scheme which, in the attempt to overcome this drawback in finite-dimensional spaces, requires the solving 
in each iteration $k \geq 0$ of the optimization problem 
\begin{equation}\label{dink-alg} \inf_{x\in S}\{f(x)-\theta_k g(x)\},
\end{equation}
while $\theta_k$ is updated by $\theta_{k+1}:=\frac{f(x^k)}{g(x^k)}$, where $x^k$ is an optimal solution of \eqref{dink-alg}. However, solving in each iteration an optimization problem of type \eqref{dink-alg} can be as expensive and difficult
as solving the fractional programming problem \eqref{fr-intr}. 

The aim of this note is to propose an alternative to this approach. Namely, we formulate two iterative schemes for solving \eqref{fr-intr}, 
where $f : {\cal H} \rightarrow \overline \R$ is proper, convex and lower semicontinuous and $g : {\cal H} \rightarrow \R$ is differentiable with Lipschitz continuous gradient and
either concave or convex. Instead of solving in each iteration \eqref{dink-alg}, the proposed iterative methods perform a gradient step 
with respect to $g$ and a proximal step with respect to $f$. In this way, the functions $f$ and $g$ are processed separately in each 
iteration. A further advantage of the algorithm investigated in case $g$ is concave comes from the fact that it generates sequences 
that concomitantly approach the set of optimal solutions and the optimal objective value of \eqref{fr-intr}. The second numerical scheme, proposed in case $g$ is convex, has the particularity that it approaches the set of critical points of the objective function of \eqref{fr-intr}, 
provided the latter satisfies the Kurdyka-\L{}ojasiewicz property.  

For the notations used in this paper we refer the reader to \cite{bo-van, bauschke-book, EkTem, Zal-carte}. Let ${\cal H}$ be a real Hilbert space 
with \textit{inner product} $\langle\cdot,\cdot\rangle$ and associated \textit{norm} $\|\cdot\|=\sqrt{\langle \cdot,\cdot\rangle}$. 
The symbols $\rightharpoonup$ and $\rightarrow$ denote weak and strong convergence, respectively. 

For a function $f:{\cal H}\rightarrow\overline{\R}$, where $\overline{\R}:=\R\cup\{\pm\infty\}$ is the extended real line, we denote
by $\dom f=\{x\in {\cal H}:f(x)<+\infty\}$ its \textit{effective domain} and say that $f$ is \textit{proper} if $\dom f\neq\emptyset$
and $f(x)\neq-\infty$ for all $x\in {\cal H}$. 
The \textit{subdifferential} of $f$ at $x\in {\cal H}$, with $f(x)\in\R$, is the set
$\partial f(x):=\{v\in {\cal H}:f(y)\geq f(x)+\langle v,y-x\rangle \ \forall y\in {\cal H}\}$.
We take by convention $\partial f(x):=\emptyset$, if $f(x)\in\{\pm\infty\}$. Let $S\subseteq {\cal H}$ be a nonempty set. 
The \textit{indicator function} of $S$, $\delta_S:{\cal H}\rightarrow \overline{\R}$, 
is the function which takes the value $0$ on $S$ and $+\infty$ otherwise.

An efficient tool for proving weak convergence of a sequence in Hilbert spaces (without a priori knowledge of its limit) is the  Opial Lemma, which we recall in the following. 

\begin{lemma}\label{opial} (Opial) Let $C$ be a nonempty set of ${\cal H}$ and $(x_k)_{k \geq 0}$ be a sequence in ${\cal H}$ such that
the following two conditions hold: \begin{itemize}\item[(a)] for every $x\in C$, $\lim_{k \rightarrow + \infty}\|x_k-x\|$ exists;
\item[(b)] every weak sequential cluster point of $(x_k)_{k \geq 0}$ is in $C$;\end{itemize}
Then $(x_k)_{k \geq 0}$ converges weakly to an element in $C$.
 \end{lemma}

When proving the first part of the Opial Lemma one usually tries to show that for every $x \in C$ the sequence $(\|x_k-x\|)_{k \geq 0}$ fulfills a Fej\'{e}r-type inequality. In this sense the following result is very useful. 

\begin{lemma}\label{Fejer-real-seq} Let $(a_k)_{k \geq 0}$, $(b_k)_{k \geq 0}$ and $(\varepsilon_k)_{k \geq 0}$ be real sequences. 
Assume that $(a_k)_{k \geq 0}$ is bounded from below, $(b_k)_{k \geq 0}$ is nonnegative, $(\varepsilon_k)_{k \geq 0} \in\ell^1$ and 
$a_{k+1}-a_k+b_k\leq\varepsilon_k$ for every $k \geq 0$. Then $(a_k)_{k \geq 0}$ is convergent and $(b_k)_{k \geq 0} \in\ell^1$.
\end{lemma}

The following summability result will be useful in Subsection 2.2. 

\begin{lemma}\label{fejer2} Let $(a_k)_{k\geq 0}$ and $(\e_k)_{k\geq 0}$ be nonnegative real sequences, such that 
$\sum_{k\geq 0}\e_k<+\infty$ and $a_{k+1}\le a\cdot a_k +\e_k$ for every $k\ge 0$, where $a\in\R$, $a<1.$ Then $\sum_{k\geq 0}a_k<+\infty.$
\end{lemma}

Finally, the descent lemma which we recall next is a helpful tool in the convergence analysis of the algorithms proposed in this manuscript. 

\begin{lemma}\label{desc-l-th} (see \cite[Lemma 1.2.3]{nes}) Let $g:{\cal H}\to\R$ be (Fr\'echet) differentiable with 
$L$-Lipschitz continuous gradient. Then 
$$g(y)\leq g(x)+\<\nabla g(x),y-x\>+\frac{L}{2}\|y-x\|^2 \ \forall x,y\in{\cal H}.$$
\end{lemma}

\section{Two proximal-gradient algorithms}

In this section we propose two proximal-gradient algorithms for solving \eqref{fr-intr} and investigate their convergence properties. We treat the situations when $g$ is either a convex or a concave function
separately.

\subsection{Concave denominator}\label{subsec21}

The problem that we investigate throughout this subsection has the following formulation.

\begin{problem}\label{pr} We are interested in solving the fractional programming problem 
\begin{equation}\label{fr-pr} \ol\theta:=\inf_{x\in S}\frac{f(x)}{g(x)}
\end{equation}
where ${\cal H}$ is a real Hilbert space, $S$ is a nonempty, convex and closed subset of ${\cal H}$, 
$C:=\left\{\ol x\in S:\ol\theta=\frac{f(\ol x)}{g(\ol x)}\right\}\neq\emptyset$ and the following conditions hold: 
\begin{align*}(H_f)& \ f:{\cal H}\rightarrow \B \mbox{ is proper, convex, lower semicontinuous such that }dom f \cap S \neq \emptyset  \mbox{ and}\\ 
& f(x)\geq 0 \ \forall x\in S;\\
(H_g)& \ g:{\cal H}\rightarrow \R \mbox{ is concave, (Fr\'{e}chet) differentiable with }L\mbox{-Lipschitz continuous gradient,}\\
& \mbox{ and there exists }M>0\mbox{ such that }0<g(x)\leq M \ \forall x\in S. \end{align*}
\end{problem}

To this aim we propose the following algorithm.

\begin{algorithm}\label{alg-fr-pr-fb}$ $

\noindent\begin{tabular}{rl}
\verb"Initialization": & \verb"Choose" $x^0\in S \cap \dom f$ and set $\theta_1:=\frac{f(x^0)}{g(x^0)}$;\\
\verb"For" $k\geq 1$ \verb"set": &  $\eta_k:=\frac{1}{2L\theta_k}$;\\
                     & $x^k:=\argmin\nolimits_{x\in S}\left[f(x)+\frac{1}{2\eta_k}\left\|x-(x^{k-1}+\theta_k\eta_k\nabla g(x^{k-1}))\right\|^2\right]$;\\
                     & $\theta_{k+1}:=\frac{f(x^k)}{g(x^k)}$.
\end{tabular}
\end{algorithm}

We are now in position to present the convergence statement of this algorithm. 

\begin{theorem}\label{conv-th} In the setting of Problem \ref{pr}, consider the sequences generated by Algorithm \ref{alg-fr-pr-fb}. The following 
statements hold: 
\begin{enumerate}
 \item[(i)] The sequence $(\theta_k)_{k\geq 1}$ is nonincreasing and $\lim_{k\rightarrow+\infty}\theta_k=\ol\theta$. Moreover, 
\begin{equation}\label{1k}0\leq \theta_{k+1}-\ol\theta \leq\frac{\theta_1(M+L\|\ol x-x^0\|^2)}{kg(\ol x)} \ \forall \ol x \in C \ \forall k \geq 1.\end{equation}
 \item[(ii)] Additionally, assume that $\inf_{x\in S}\frac{f(x)}{g(x)} > 0$. Then the sequence 
 $(x^k)_{k\geq 0}$ converges weakly to an element in $C$. 
\end{enumerate}
\end{theorem}

\begin{proof} According to the first order optimality conditions we have
\begin{equation}\label{opt-cond}0\in\partial (f+\delta_S)(x^k)+\frac{1}{\eta_k}\left(x^k-x^{k-1}-\theta_k\eta_k\nabla g(x^{k-1})\right) \ \forall k\geq 1.\end{equation}
A direct consequence of the definition of the convex subdifferential is the inequality 
\begin{align}\label{f-conv} f(x)-f(x^k)\geq & \ \<\frac{1}{\eta_k}(x^{k-1}-x^k)+\theta_k\nabla g(x^{k-1}),x-x^k\>\nonumber\\
= & \ \frac{1}{2\eta_k}\left(\|x^{k-1}-x^k\|^2+\|x-x^k\|^2-\|x-x^{k-1}\|^2\right)\nonumber\\
& +\theta_k\langle \nabla g(x^{k-1}),x-x^k\rangle \ \forall x\in S \ \forall k\geq 1. 
\end{align}
Invoking the concavity of $g$ and using that $\theta_k\geq 0$, we have 
\begin{equation}\label{g-conc}-\theta_k g(x)+\theta_kg(x^{k-1})
\geq \theta_k\langle \nabla g(x^{k-1}),x^{k-1}-x\rangle \ \forall x\in S \ \forall k\geq 1.\end{equation}
Combining \eqref{f-conv} and \eqref{g-conc}, we obtain 
\begin{align}\label{ineq1} f(x)-\theta_k g(x)\geq & \ \frac{1}{2\eta_k}\left(\|x^{k-1}-x^k\|^2+\|x-x^k\|^2-\|x-x^{k-1}\|^2\right)\nonumber\\
& \ + f(x^k)-\theta_kg(x^{k-1})+\theta_k\langle \nabla g(x^{k-1}),x^{k-1}-x^k\rangle \ \forall x\in S \ \forall k\geq 1.
\end{align}
Lemma \ref{desc-l-th} applied to the function $-g$ yields the inequality 
$$-\theta_kg(x^{k-1})+\theta_k\langle \nabla g(x^{k-1}),x^{k-1}-x^k\rangle\geq
-\theta_k g(x^k)-\frac{L\theta_k}{2}\|x^k-x^{k-1}\|^2 \ \forall k\geq 1,$$
hence from \eqref{ineq1} we derive 
\begin{align*}f(x)-\theta_k g(x)\geq & \ \frac{1}{2\eta_k}\left(\|x^{k-1}-x^k\|^2+\|x-x^k\|^2-\|x-x^{k-1}\|^2\right)\\
& \ + f(x^k)-\theta_kg(x^k)-\frac{L\theta_k}{2}\|x^k-x^{k-1}\|^2 \ \forall x\in S \ \forall k\geq 1.
\end{align*}
Taking into account the relation $f(x^k)=\theta_{k+1}g(x^k)$ and the way $\eta_k$ is defined, we obtain for every $x\in S$ and $k\geq 1$ the 
inequality 
\begin{equation}\label{ineq2}f(x)-\theta_k g(x)\geq(\theta_{k+1}-\theta_k)g(x^k)+\frac{L\theta_k}{2}\|x^k-x^{k-1}\|^2
+L\theta_k\|x-x^k\|^2-L\theta_k\|x-x^{k-1}\|^2.\end{equation}

(i) Taking $x:=x^{k-1}$ in \eqref{ineq2} we get 
\begin{equation}\label{ineq3}(\theta_{k+1}-\theta_k)g(x^k)+\frac{3L\theta_k}{2}\|x^k-x^{k-1}\|^2\leq 0 \ \forall k\geq 1.\end{equation}
This further implies that $(\theta_k)_{k\geq 1}$ is a nonincreasing sequence, hence convergent, since it is bounded from below by $0$. 

Consider now an arbitrary $\ol x\in C$ and take $x:=\ol x$ in \eqref{ineq2}. We derive 
\begin{align} \frac{L\theta_k}{2}\|x^k-x^{k-1}\|^2+L\theta_k\|\ol x -x^k\|^2-L\theta_k\|\ol x -x^{k-1}\|^2 &\leq (\theta_k-\theta_{k+1})M+f(\ol x)-\theta_k g(\ol x)\nonumber\\
& = (\theta_k-\theta_{k+1})M+(\ol \theta-\theta_k) g(\ol x) \nonumber \\
&\leq (\theta_k-\theta_{k+1})M \ \forall k\geq 1 \label{ineq4}.
\end{align}
This yields the inequality 
\begin{equation}\label{ineq5} \frac{L\theta_k}{2}\|x^k-x^{k-1}\|^2+L\theta_{k+1}\|\ol x-x^k\|^2-L\theta_k\|\ol x-x^{k-1}\|^2
\leq (\theta_k-\theta_{k+1})M \ \forall k\geq 1.
\end{equation}
Since $(\theta_k)_{k\geq 1}$ is bounded from below by $0$, the sequence on the right-hand side of inequality \eqref{ineq5}  belongs to $\ell^1$. 
We derive from Lemma \ref{Fejer-real-seq} that 
\begin{equation}\label{theta-xn-l2}\sum_{k\geq 1}\theta_k\|x^k-x^{k-1}\|^2<+\infty
\end{equation}
and 
\begin{equation}\label{theta-op1}(\theta_{k+1}\|\ol x-x^k\|^2)_{k\geq 0}\mbox{ is convergent}.
\end{equation}
Coming back to \eqref{ineq4} and using $\theta_k\geq \theta_{k+1}$, we obtain
\begin{equation}\label{ineq5'}
(\ol \theta-\theta_k)g(\ol x)\geq (\theta_{k+1}-\theta_k)M+\frac{L\theta_k}{2}\|x^k-x^{k-1}\|^2+L\theta_{k+1}\|\ol x -x^k\|^2-L\theta_k\|\ol x -x^{k-1}\|^2 \ \forall k \geq 1.
\end{equation}
Relying on  \eqref{theta-xn-l2} and \eqref{theta-op1} and the convergence of the sequence $(\theta_k)_{k\geq 1}$, the right-hand side 
of the above inequality is a sequence which converges to $0$ as $k\rightarrow+\infty$. Invoking also the fact that $(\theta_k)_{k\geq 1}$ 
is bounded from below by $\ol \theta$, we conclude  that $\lim_{k\rightarrow+\infty}\theta_k=\ol\theta$.

Let us prove now the convergence rate result stated in \eqref{1k}. Let $\ol x\in C$ and $n\geq 1$ be arbitrary. 
From \eqref{ineq3} we obtain $$\sum_{k=1}^n(k-1)(\theta_k-\theta_{k+1})g(x^k)\geq \frac{3}{2}L\sum_{k=1}^n(k-1)\theta_k\|x^k-x^{k-1}\|^2,$$
hence 
\begin{align*} \sum_{k=1}^n\Big((k-1)\theta_k-k\theta_{k+1}+\theta_{k+1}\Big)=& \sum_{k=1}^n(k-1)(\theta_k-\theta_{k+1})\\
\geq & \frac{3L}{2M}\sum_{k=1}^n(k-1)\theta_k\|x^k-x^{k-1}\|^2.
\end{align*}
Noticing the telescoping sum in the left-hand side of the previous inequality, we obtain
\begin{equation}\label{ineq6} -n\theta_{n+1}+\sum_{k=1}^n\theta_{k+1}\geq \frac{3L}{2M}\sum_{k=1}^n(k-1)\theta_k\|x^k-x^{k-1}\|^2.
\end{equation}

Summing up the inequalities in \eqref{ineq5'} for $k$ from $1$ to $n+1$ we obtain 
\begin{align*}
& \left((n+1)\ol\theta-\sum_{k=1}^{n+1}\theta_k\right)g(\ol x)\geq \\ 
& (\theta_{n+2}-\theta_1)M+\frac{L}{2}\sum_{k=1}^{n+1}\theta_k\|x^k-x^{k-1}\|^2 +L\theta_{n+2}\|\ol x-x^{n+2}\|^2-L\theta_1\|\ol x-x^0\|^2,
\end{align*}
hence 
\begin{align}\label{ineq7}
& (n+1)\ol\theta-\sum_{k=1}^{n}\theta_{k+1}-\theta_1\geq \nonumber\\ 
& \frac{1}{g(\ol x)}\left[(\theta_{n+2}-\theta_1)M+\frac{L}{2}\sum_{k=1}^{n+1}\theta_k\|x^k-x^{k-1}\|^2 +L\theta_{n+2}\|\ol x-x^{n+2}\|^2-L\theta_1\|\ol x-x^0\|^2\right].
\end{align}

Summing up the inequalities \eqref{ineq6} and \eqref{ineq7} and discarding the nonnegative terms on the right-hand side we derive 
$$n(\ol\theta-\theta_{n+1})+\ol\theta-\theta_1\geq -\frac{\theta_1 M}{g(\ol x)}-\frac{L\theta_1}{g(\ol x)}\|\ol x-x^0\|^2.$$
Noticing that $\theta_1\geq \ol\theta$, the last inequality implies \eqref{1k} after rearranging the terms. 

(ii) For the remaining of the proof we assume that $\inf_{x\in S}\frac{f(x)}{g(x)} > 0$. In this situation, $\lim_{k\rightarrow+\infty}\theta_k=\ol\theta>0$ and from \eqref{theta-xn-l2} and \eqref{theta-op1} 
we derive \begin{equation}\label{k-k-1}\lim_{k\rightarrow+\infty}(x^k-x^{k-1})=0\end{equation}
and \begin{equation}\label{op1}(\|\ol x-x^k\|)_{k \geq 1}\mbox{ is convergent }\forall \ol x\in C.\end{equation}
Thus the first condition in the Opial Lemma is fulfilled. 

From \eqref{opt-cond} we deduce 
$$\frac{1}{\eta_{k_l}}(x^{k_l-1}-x^{k_l})+\theta_{k_l}\nabla g(x^{{k_l}-1})\in\partial (f+\delta_S)(x^{k_l}),$$
hence 
\begin{equation}\label{opt-cond2}\frac{1}{\eta_{k_l}}(x^{{k_l}-1}-x^{k_l})+\theta_{k_l}\nabla g(x^{{k_l}-1})-\ol\theta\nabla g(x^{k_l})\in
\partial (f+\delta_S-\ol\theta g)(x^{k_l}) \ \forall l\geq 1,\end{equation}
due to the concavity of $g$ and $\ol\theta>0$. Since for every $l \geq 1$ we have 
\begin{align*}\|\theta_{k_l} \nabla g(x^{{k_l}-1})-\ol\theta\nabla g(x^{k_l})\|\leq & \ |\theta_{k_l}-\ol\theta|\|\nabla g(x^{{k_l}-1})\|+\ol\theta\|\nabla g(x^{{k_l}-1})-\nabla g(x^{k_l})\|\\
\leq &  L |\theta_{k_l}-\ol\theta|\|x^{{k_l}-1}- x^0\|\!+ \!|\theta_{k_l}-\ol\theta| \|\nabla g(x^0)\|\!+\!\ol\theta L\!\|x^{{k_l}-1}-x^{k_l}\|,
\end{align*}
from (i), \eqref{k-k-1} and the fact that $(x^{k_l})_{l \geq 0}$ is bounded, we conclude that 
$$\frac{1}{\eta_{k_l}}(x^{{k_l}-1}-x^{k_l})+\theta_{k_l}\nabla g(x^{{k_l}-1})-\ol\theta\nabla g(x^{k_l})\rightarrow 0 \mbox{ as }l\rightarrow+\infty.$$
Noticing that $(x^{k_l})_{l \geq 0}$ converges weakly to $\ol x$ as $l\rightarrow+\infty$, from \eqref{opt-cond2} and the fact that the graph 
of the convex subdifferential of a proper, convex and lower semicontinuous function is sequentially closed with respect to the 
weak-norm topology (see \cite[Proposition 20.33]{bauschke-book}), we derive that $$0\in\partial (f+\delta_S-\ol\theta g)(\ol x),$$
hence $\ol x\in\dom f\cap S$. 
The definition of the convex subdifferential yields the inequality $$f(y)-\ol\theta g(y)\geq f(\ol x)-\ol\theta g(\ol x) \ \forall y\in S.$$
From here, by choosing $y \in C$, we get 
\begin{align*}0 \geq  \ f(\ol x)-\ol\theta g(\ol x),
\end{align*}
hence \begin{equation}\label{ineq8}\ol\theta\geq \frac{f(\ol x)}{g(\ol x)}.\end{equation}
Relation \eqref{ineq8} implies now that $\ol x\in C$. Thus the second condition in the Opial Lemma is also fulfilled.
The conclusion follows now from Lemma \ref{opial}. 
\end{proof}

\subsection{Convex denominator}\label{subsec22}

In this subsection we consider the case when $g$ is a convex function. 

\begin{problem}\label{pr2} We are interested in solving the fractional programming problem 
\begin{equation}\label{fr-pr2} \inf_{x\in S}\frac{f(x)}{g(x)}
\end{equation}
where ${\cal H}$ is a real Hilbert space, $S$ is a nonempty, convex and closed subset of ${\cal H}$, 
and the following conditions hold: 
\begin{align*}(H_f)& \ f:{\cal H}\rightarrow \B \mbox{ is proper, convex, lower semicontinuous such that }dom f \cap S \neq \emptyset  \mbox{ and}\\ 
& f(x)\geq 0 \ \forall x\in S;\\
(\widetilde H_g)& \ g:{\cal H}\rightarrow \R \mbox{ is convex, continuously (Fr\'{e}chet) differentiable and there exists }M>0\\
& \mbox{such that }0<g(x)\leq M \ \forall x\in S. \end{align*}
\end{problem}

The algorithm we propose in this context has the following formulation. 

\begin{algorithm}\label{alg-fr-pr-fb2}$ $

\noindent\begin{tabular}{rl}
\verb"Initialization": & \verb"Choose" $x^0\in S\cap\dom f$ and set $\theta_1:=\frac{f(x^0)}{g(x^0)}$;\\
\verb"For" $k\geq 1$ \verb"do": &  \verb"Choose" $\eta_k>0$;\\
                     &  \verb"Set" $x^k:=\argmin\nolimits_{x\in S}\left[f(x)+\frac{1}{2\eta_k}\left\|x-(x^{k-1}+\theta_k\eta_k\nabla g(x^{k-1}))\right\|^2\right]$;\\
                     &  \verb"Set" $\theta_{k+1}:=\frac{f(x^k)}{g(x^k)}$.
\end{tabular}
\end{algorithm}

The proof of the first result in this subsection reveals the fact that when $g$ is convex one cannot expect convergence of the whole 
sequence $(x^k)_{k\geq 0}$. Furthermore, if this is the case, then the limit is not necessarily an optimal solution of \eqref{fr-pr2}, 
but a critical point of the objective function $\frac{f+\delta_S}{g}$ in the sense of the limiting subdifferential. In order to explain 
this notion, we need some prerequisites of nonsmooth analysis. 

For the following generalized subdifferential notions and their basic properties we refer to \cite{boris-carte, rock-wets}. 
Let $h:{\cal H}\rightarrow \R\cup\{+\infty\}$ be a proper and lower semicontinuous function. If $x\in\dom h$, we consider the {\it Fr\'{e}chet (viscosity)  
subdifferential} of $h$ at $x$ as being the set $$\hat{\partial}h(x):= \left \{v\in{\cal H}: \liminf_{y\rightarrow x}\frac{h(y)-h(x)-\<v,y-x\>}{\|y-x\|}\geq 0 \right \}.$$ For 
$x\notin\dom h$ we set $\hat{\partial}h(x):=\emptyset$. The {\it limiting (Mordukhovich) subdifferential} is defined at $x\in \dom h$ by 
$$\partial_L h(x):=\{v\in{\cal H}:\exists x_k\rightarrow x,h(x_k)\rightarrow h(x)\mbox{ and }\exists v_k\in\hat{\partial}h(x_k),v_k\rightharpoonup v \mbox{ as }k\rightarrow+\infty\},$$
while for $x \notin \dom h$, one takes $\partial_L h(x) :=\emptyset$. Therefore  $\hat\partial h(x)\subseteq\partial_L h(x)$ for each $x\in{\cal H}$.

When $h$ is continuously differentiable around $x \in {\cal H}$ we have $\partial_L h(x)=\{\nabla h(x)\}$. 
Notice that in case $h$ is convex, these two subdifferential notions coincide with the {\it convex subdifferential}, thus 
$\hat\partial h(x)=\partial_L h(x)=\{v\in{\cal H}:h(y)\geq h(x)+\<v,y-x\> \ \forall y\in {\cal H}\}$ for all $x\in{\cal H}$. 

The Fermat rule reads in this nonsmooth setting: if $x\in{\cal H}$ is a local minimizer of $h$, then $0\in\partial_L h(x)$. An element 
$x\in\dom h$ fulfilling this inclusion relation is called {\it critical point} of the function $h$. The set of all critical points of $h$ 
is denoted by $\crit(h)$. 

The convergence of Algorithm \eqref{alg-fr-pr-fb2} is stated in the following theorem.

\begin{theorem}\label{conv-th2} In the setting of Problem \ref{pr2}, consider the sequences generated by Algorithm \ref{alg-fr-pr-fb2} such that the additional condition $\liminf_{k\rightarrow+\infty}\eta_k>0$ is satisfied. The following statements hold: 
\begin{enumerate}
 \item[(i)] The sequence $(\theta_k)_{k\geq 1}$ is nonincreasing, hence convergent. Moreover, 
 $$\sum_{k\geq 1}\frac{1}{\eta_k}\|x^k-x^{k-1}\|^2<+\infty.$$
 \item[(ii)] For every (strong) limit point $\ol x$ of $(x^k)_{k \geq 0}$, it holds 
 $\ol x\in\dom f\cap S$ and $\lim_{k\rightarrow+\infty}\theta_k=\frac{f(\ol x)}{g(\ol x)}$. 
 If we additionally have that $\ol x\in \inte(\dom f \cap S)$, then $0\in\partial_L \left(\frac{f+\delta_S}{g}\right)(\ol x)$. 
\end{enumerate}
\end{theorem}

\begin{proof} As already seen in the proof of Theorem \ref{conv-th}, we have
\begin{equation}\label{opt-cond-pr2}\frac{1}{\eta_k}(x^{k-1}-x^k)+\theta_k\nabla g(x^{k-1})\in\partial (f+\delta_S)(x^k) \ \forall k\geq 1,\end{equation}
and \begin{align}\label{f-conv2} f(x)-f(x^k)\geq & \ \frac{1}{2\eta_k}\left(\|x^{k-1}-x^k\|^2+\|x-x^k\|^2-\|x-x^{k-1}\|^2\right)\nonumber\\
& \ +\theta_k\langle \nabla g(x^{k-1}),x-x^k\rangle \ \forall x\in S \ \forall k\geq 1. 
\end{align}
By choosing $x:=x^{k-1}$ in \eqref{f-conv2} we obtain 
$$f(x^{k-1})-f(x^k)\geq \frac{1}{\eta_k}\|x^{k-1}-x^k\|^2+\theta_k\langle \nabla g(x^{k-1}),x^{k-1}-x^k\rangle \ \forall k\geq 1.$$
Further, by combining this with 
$$\theta_k\Big(g(x^k)-g(x^{k-1})\Big)\geq \theta_k\langle \nabla g(x^{k-1}),x^k-x^{k-1}\rangle,$$
we obtain \begin{align}0=f(x^{k-1})-\theta_kg(x^{k-1}) \geq & \ f(x^k)-\theta_k g(x^k) + \frac{1}{\eta_k}\|x^{k-1}-x^k\|^2\nonumber\\
                        = & \ (\theta_{k+1}-\theta_k)g(x^k)+\frac{1}{\eta_k}\|x^{k-1}-x^k\|^2 \ \forall k\geq 1\label{ineq9}.
          \end{align}
(i) From \eqref{ineq9} we obtain that $(\theta_k)_{k\geq 1}$ is nonincreasing, hence convergent, since it is bounded from below by $0$. 
Moreover, from \eqref{ineq9} we obtain $$\frac{1}{\eta_k}\|x^{k-1}-x^k\|^2\leq (\theta_k-\theta_{k+1})M \ \forall k\geq 1,$$
hence $\sum_{k\geq 1}\frac{1}{\eta_k}\|x^k-x^{k-1}\|^2<+\infty.$

(ii) Assume that $x^{k} \rightarrow \ol x$ as $k \rightarrow +\infty$. Since $S$ is closed, we have $\ol x\in S$. By choosing  $x:=\ol x$ in \eqref{f-conv2}, we obtain  
\begin{align*} f(\ol x)-f(x^{k})\geq & \ \frac{1}{2\eta_{k}}\left(\|x^{{k}-1}-x^{k}\|^2+\|\ol x-x^{k}\|^2-\|\ol x-x^{{k}-1}\|^2\right)\nonumber\\
& \ +\theta_{k}\langle \nabla g(x^{{k}-1}),\ol x-x^{k}\rangle  \ \forall k\geq 1. 
\end{align*}
By using (i), one can see that the right-hand side of the above inequality converges to $0$ as $k\rightarrow+\infty$. Hence, 
$\limsup_{k\rightarrow+\infty}f(x^k)\leq f(\ol x)$. Since $f$ is lower semicontinuous, the reverse inequality is also true, thus 
$$\lim_{k\rightarrow+\infty}f(x^k) = f(\ol x).$$ Furthermore, due to the continuity of $g$, we have 
$$\lim_{k\rightarrow+\infty}g(x^k) = g(\ol x).$$ Let us denote by $\theta$ the limit of the sequence $(\theta_k)_{k\geq 1}$. Passing 
to the limit as $k\rightarrow+\infty$ in the relation which defines $\theta_{k+1}$ in Algorithm \ref{alg-fr-pr-fb2}, we obtain
\begin{equation}\label{theta-f-g}\theta=\frac{f(\ol x)}{g(\ol x)}.\end{equation}
By using again the closedness property of the graph of the convex subdifferential, from \eqref{opt-cond-pr2} and (i) we obtain
\begin{equation}\label{opt-cond-sec2}\theta\nabla g(\ol x)\in\partial (f+\delta_S)(\ol x),\end{equation}
hence $\ol x\in \dom f\cap S$. 

Assume now that $\ol x\in \inte(\dom f \cap S)$. In this situation $f+\delta_S$ is Lipschitz continuous around $\ol x$ 
(see \cite[Theorem 8.29]{bauschke-book}). From \eqref{theta-f-g} and \eqref{opt-cond-sec2} we obtain
\begin{align*} 0& \ \in -\frac{(f+\delta_S)(\ol x)}{g^2(\ol x)}\nabla g(\ol x)+\frac{1}{g(\ol x)}\partial (f+\delta_S)(\ol x)\\
 & \ = (f+\delta_S)(\ol x)\nabla \left(\frac{1}{g}\right)(\ol x)+\partial \left(\frac{1}{g(\ol x)}(f+\delta_S)\right)(\ol x)\\
 & \ = (f+\delta_S)(\ol x)\nabla \left(\frac{1}{g}\right)(\ol x)+\partial_L \left(\frac{1}{g(\ol x)}(f+\delta_S)\right)(\ol x)\\
 & \ = \partial_L \left(\frac{1}{g}\cdot(f+\delta_S)\right)(\ol x),
\end{align*}
where the last equality makes use of \cite[Corollary 1.111(i)]{boris-carte}. 
\end{proof}

\begin{remark} (a) The main ingredient in the proof of the second statement of the above theorem is the  rule for the limiting subdifferential of the product (or quotient) of locally Lipschitz continuous functions. We notice that similar rules are valid also 
for the Clarke subdifferential (see \cite[Exercise 10.21]{clarke}). 

(b) Whenever $\frac{f+\delta_S}{g}$ is a convex function, we obtain in the hypotheses of the above theorem  that $\ol x$ is a global optimal solution of \eqref{fr-pr2} and 
$\lim_{k\rightarrow+\infty}\theta_k=\frac{f(\ol x)}{g(\ol x)}=\inf_{x\in S}\frac{f(x)}{g(x)}$. 
\end{remark}

In the remaining of this subsection we address the question whether one can guarantee the convergence of the whole sequence 
$(x^k)_{k\geq 0}$ generated in Algorithm \ref{alg-fr-pr-fb2}. We will see that this is ensured whenever the objective function of \eqref{fr-pr2} 
satisfies the \textit{Kurdyka-\L{}ojasiewicz property}. To this end we recall some notations and definitions related to the latter. 

For the remaining of this section we suppose that ${\cal H}$ is finite dimensional. For $\eta\in(0,+\infty]$, we denote by $\Theta_{\eta}$ the class of concave and continuous functions 
$\varphi:[0,\eta)\rightarrow [0,+\infty)$ such that $\varphi(0)=0$, $\varphi$ is continuously differentiable on $(0,\eta)$, continuous at $0$ and $\varphi'(s)>0$ for all 
$s\in(0, \eta)$. In the following definition (see \cite{att-b-red-soub2010, b-sab-teb}) we use also the {\it distance function} to a set, 
defined for $A\subseteq{\cal H}$ as $\dist(x,A)=\inf_{y\in A}\|x-y\|$ for all $x\in{\cal H}$. 

\begin{definition}\label{KL-property} \rm({\it Kurdyka-\L{}ojasiewicz property}) Let $h:{\cal H}\rightarrow\B$ be a proper and lower 
semicontinuous function. We say that $h$ satisfies the {\it Kurdyka-\L{}ojasiewicz (KL) property} at 
$\ol x\in \dom\partial_L h=\{x\in{\cal H}:\partial_L h(x)\neq\emptyset\}$ 
if there exists $\eta \in(0,+\infty]$, a neighborhood $U$ of $\ol x$ and a function $\varphi\in \Theta_{\eta}$ such that for all $x$ in the 
intersection 
$$U\cap \{x\in{\cal H}: h(\ol x)<h(x)<h(\ol x)+\eta\}$$ the following inequality holds 
$$\varphi'(h(x)-h(\ol x))\dist(0,\partial_L h(x))\geq 1.$$
If $h$ satisfies the KL property at each point in $\dom\partial h$, then $h$ is called a {\it KL function}. 
\end{definition}

The origins of this notion go back to the pioneering work of \L{}ojasiewicz \cite{lojasiewicz1963}, where it is proved that for a real-analytic function 
$h:{\cal H}\rightarrow\R$ and a critical point $\ol x\in{\cal H}$ (that is $\nabla h(\ol x)=0$), there exists $\theta\in[1/2,1)$ such that the function 
$|h-h(\ol x)|^{\theta}\|\nabla h\|^{-1}$ is bounded around $\ol x$. This corresponds to the situation when $\varphi(s)=Cs^{1-\theta}$, where 
$C>0$. The result of 
\L{}ojasiewicz allows the interpretation of the KL property as a re-parametrization of the function values in order to avoid flatness around the 
critical points. Kurdyka \cite{kurdyka1998} extended this property to differentiable functions definable in an o-minimal structure. 
Further extensions to the nonsmooth setting can be found in \cite{b-d-l2006, att-b-red-soub2010, b-d-l-s2007, b-d-l-m2010}. 

One of the remarkable properties of KL functions is their ubiquity in applications, according to \cite{b-sab-teb}. To this class of functions belong semi-algebraic, real sub-analytic, semiconvex, uniformly convex and 
convex functions satisfying a growth condition. We refer the reader to 
\cite{b-d-l2006, att-b-red-soub2010, b-d-l-m2010, b-sab-teb, b-d-l-s2007, att-b-sv2013, attouch-bolte2009} and the references therein  for more details regarding KL functions and illustrating examples. 

An important role in our convergence analysis will be played by the following uniformized KL property given in \cite[Lemma 6]{b-sab-teb}. 

\begin{lemma}\label{unif-KL-property} Let $\Omega\subseteq {\cal H}$ be a compact and connected set and let $h:{\cal H}\rightarrow\B$ be a proper 
and lower semicontinuous function. Assume that $h$ is constant on $\Omega$ and $h$ satisfies the KL property at each point of $\Omega$.   
Then there exist $\varepsilon,\eta >0$ and $\varphi\in \Theta_{\eta}$ such that for all $\ol x\in\Omega$ and for all $x$ in the intersection 
\begin{equation}\label{int} \{x\in{\cal H}: \dist(x,\Omega)<\varepsilon\}\cap \{x\in{\cal H}: h(\ol x)<h(x)<h(\ol x)+\eta\}\end{equation} 
the following inequality holds \begin{equation}\label{KL-ineq}\varphi'(h(x)-h(\ol x))\dist(0,\partial_L h(x))\geq 1.\end{equation}
\end{lemma}

The techniques used below are well-known in the community dealing with algorithms for optimization problems involving functions with the 
Kurdyka-\L{}ojasiewicz property (see \cite{b-sab-teb, bcl, att-b-sv2013, c-pesquet-r}). We show that this approach can be used also 
for fractional programming problems. 

In the following we denote by $\omega((x^k)_{k\geq 0})$ the set of cluster points of the sequence $(x^k)_{k\geq 0}$. The first statement 
in the next result is a direct consequence of Theorem \ref{conv-th2}, while the other statements can be proved similar to 
\cite[Lemma 5]{b-sab-teb}, where it is noticed that (b) and (c) are generic for sequences satisfying the relation 
$\lim_{k\rightarrow+\infty}(x^k-x^{k-1})=0$. 

\begin{lemma}\label{l7} In the setting of Problem \ref{pr2}, let ${\cal H}$ be finite dimensional and consider the sequences generated 
by Algorithm \ref{alg-fr-pr-fb2} such that the additional condition 
$$0<\liminf_{k\rightarrow+\infty}\eta_k\leq \limsup_{k\rightarrow+\infty}\eta_k<+\infty$$ is satisfied. Assume that $(x^k)_{k\geq 0}$ is 
bounded. The following statements hold: 
\begin{itemize}
\item[(a)] $\omega((x^k)_{k\geq 0})\cap\inte(\dom f\cap S) \subseteq \crit\left(\frac{f+\delta_S}{g}\right)$; 
\item[(b)] $\lim_{k\to\infty}\dist\Big(x^k,\omega((x^k)_{k\geq 0})\Big)=0$;
\item[(c)] $\omega((x^k)_{k\geq 0})$ is nonempty, compact and connected;
\item[(d)] $\frac{f+\delta_S}{g}$ is finite and constant on $\omega((x^k)_{k\geq 0})$.
\end{itemize}
\end{lemma}

\begin{remark} Suppose that $\frac{f+\delta_S}{g}$ is coercive, that is 
$$\lim_{\|u\|\rightarrow+\infty}\left(\frac{f+\delta_S}{g}\right)(u)=+\infty.$$  Then the sequence $(x^k)_{k\geq 0}$ generated by Algorithm \ref{alg-fr-pr-fb2} is bounded. Indeed, this follows from the fact that 
$(\theta_k)_{k\geq 1}$ is nonincreasing and the lower level sets of $\frac{f+\delta_S}{g}$ are bounded. 
\end{remark}

We give now the main result concerning the convergence of the whole sequence $(x^k)_{k\geq 0}$. 

\begin{theorem}\label{t1} In the setting of Problem \ref{pr2}, let ${\cal H}$ be finite dimensional, $\nabla g$ be $L$-Lipschitz continuous, 
and consider the sequences generated by Algorithm \ref{alg-fr-pr-fb2} under the additional conditions  
$\liminf_{k\rightarrow+\infty}\eta_k>0$, $\eta_1\theta_1<\frac{1}{L}$ and $(\eta_k)_{k\geq 1}$ nonincreasing. Assume that  
$\frac{f+\delta_S}{g}$ is a KL function. Moreover, suppose that $(x^k)_{k\geq 0}$ is 
bounded and there exists $k_0\geq 0$ such that 
$x^k\in\inte(\dom f\cap S)$ for all $k\geq k_0$. Then the following statements are true:\begin{itemize}
 \item[(a)] $\sum_{k\geq 0}\|x^{k+1}-x^k\|<+\infty$;
 \item[(b)] there exists $ x_{\infty} \in\dom f \cap S$ such that $\lim_{k\rightarrow+\infty}x^k=x_{\infty}$. If additionally 
 $x_{\infty} \in\inte(\dom f\cap S)$, then $0\in\partial_L \left(\frac{f+\delta_S}{g}\right)( x_{\infty})$.
\end{itemize}
\end{theorem}

\begin{proof} (a) Consider the sequences generated by Algorithm \ref{alg-fr-pr-fb2}. 
According to Lemma \ref{l7} we can choose an element $\ol x \in \omega((x^k)_{k\geq 0})$. 
By Theorem \ref{conv-th2}(ii), we have $\ol x\in\dom f \cap S$ and $\lim_{k\rightarrow +\infty}\theta_k=\frac{f(\ol x)}{g(\ol x)}$. 
We separately treat the following two cases. 

I. There exists $\ol k\geq 1$ such that $\theta_{\ol k}=\frac{f(\ol x)}{g(\ol x)}$. Since $(\theta_k)_{k\geq 1}$ is nonincreasing, we have 
$\theta_{k}=\frac{f(\ol x)}{g(\ol x)}$ for every $k\geq \ol k$. By using \eqref{ineq9}, we deduce that the sequence 
$(x^k)_{k\geq \ol k}$ is constant. From here the conclusion follows automatically. 

II. For all $k\geq 1$ it holds $\theta_k>\frac{f(\ol x)}{g(\ol x)}$. Take $\Omega:=\omega ((x^k)_{k\geq 0})$.

In virtue of Lemma \ref{l7}(c) and (d) and Lemma \ref{unif-KL-property}, the KL property of $\frac{f+\delta_S}{g}$ leads to the existence of positive numbers $\e$ and $\eta$ and 
a concave function $\varphi\in\Phi_{\eta}$ such that for all
\begin{equation}\label{int-H} 
x\in  \{u\in {\cal H}: \dist(u,\Omega)<\e\}
  \cap\left\{u\in\dom f\cap S:\frac{f(\ol x)}{g(\ol x)}<\frac{f(u)}{g(u)}<\frac{f(\ol x)}{g(\ol x)}+\eta\right\}\end{equation}
one has
\begin{equation}\label{ineq-H}\varphi'\left(\frac{f(x)}{g(x)}-\frac{f(\ol x)}{g(\ol x)}\right)\dist(0,\partial_L f(x))\ge 1.\end{equation}

Let $k_1\geq 0$ be such that $\theta_k<\frac{f(\ol x)}{g(\ol x)}+\eta$ for all $k\geq k_1.$  
According to Lemma \ref{l7}(b), there exists $k_2\geq 0$ such that $\dist(x^k,\Omega)<\e$ for all $k\geq k_2.$

Hence the sequence $(x^k)_{k\geq \ol k}$, where $\ol k=\max\{k_1,k_2\}$, belongs to the intersection \eqref{int-H}. So we have (see \eqref{ineq-H})
\begin{equation}\label{ineq-H'}\varphi'\left(\frac{f(x^k)}{g(x^k)}-\frac{f(\ol x)}{g(\ol x)}\right)\|x^*\|\ge 1
\ \forall x^*\in\partial_L\left(\frac{f+\delta_S}{g}\right)(x^k) \ \forall k\geq\ol k.\end{equation}
Since $\varphi$ is concave, it holds for all $x^*\in\partial_L\left(\frac{f+\delta_S}{g}\right)(x^k)$ and for all $k\geq \ol k$
\begin{align}
\e_k & := \varphi\left(\frac{f(x^k)}{g(x^k)}-\frac{f(\ol x)}{g(\ol x)}\right)-\varphi\left(\frac{f(x^{k+1})}{g(x^{k+1})}-\frac{f(\ol x)}{g(\ol x)}\right)\\ 
& \ge \varphi'\left(\frac{f(x^k)}{g(x^k)}-\frac{f(\ol x)}{g(\ol x)}\right)\cdot\left(\frac{f(x^k)}{g(x^k)}-\frac{f(x^{k+1})}{g(x^{k+1})}\right)\nonumber\\ 
 & = \varphi'\left(\frac{f(x^k)}{g(x^k)}-\frac{f(\ol x)}{g(\ol x)}\right)\cdot\left(\theta_{k+1}-\theta_{k+2}\right)\nonumber\\
 & \ge \frac{1}{\|x^*\|}\cdot\left(\theta_{k+1}-\theta_{k+2}\right)\nonumber\\
 & \geq \frac{1}{\|x^*\|}\cdot\frac{1}{\eta_{k+1}g(x^{k+1})}\|x^{k+1}-x^k\|^2,\label{ineq10}
\end{align}
where the last inequality follows from \eqref{ineq9}. 

Further, by using \eqref{opt-cond-pr2} and \cite[Corollary 1.111(i)]{boris-carte}, we have that for every $k \geq k_0$
$$x_k^*:=-\frac{\nabla g(x^k)}{g^2(x^k)}(f+\delta_S)(x^k)+\frac{1}{g(x^k)}\left[\frac{1}{\eta_k}(x^{k-1}-x^k)+\theta_k\nabla g(x^{k-1})\right]
\in\partial_L\left(\frac{f+\delta_S}{g}\right)(x^k).$$
Furthermore, notice that $$x_k^*=\frac{1}{g(x^k)}\left[-\theta_{k+1}\nabla g(x^k)+\theta_k\nabla g(x^{k-1})+\frac{1}{\eta_k}(x^{k-1}-x^k)\right]$$
and relying on the Lipschitz continuity of the gradient we derive 
$$\|x_k^*\|\leq \frac{1}{g(x^k)}\left[\left(\frac{1}{\eta_k}+\theta_k L\right)\|x^k-x^{k-1}\|+(\theta_k-\theta_{k+1})\|\nabla g(x^k)\|\right].$$

Altogether, from \eqref{ineq10} we obtain for every $k\geq \max\{\ol k,k_0\}$
$$\e_k\geq \frac{g(x^k)}{g(x^{k+1})}\cdot\frac{\frac{1}{\eta_{k+1}}\|x^{k+1}-x^k\|^2}{\left(\frac{1}{\eta_k}+
\theta_k L\right)\|x^k-x^{k-1}\|+(\theta_k-\theta_{k+1})\|\nabla g(x^k)\|} $$ and from here
\begin{align} \|x^{k+1}-x^k\|\leq & \ \sqrt{\eta_{k+1}\left[\left(\frac{1}{\eta_k}+
\theta_k L\right)\|x^k-x^{k-1}\|+(\theta_k-\theta_{k+1})\|\nabla g(x^k)\|\right]\frac{g(x^{k+1})}{g(x^k)}\e_k}\nonumber\\
\leq & \ \frac{\eta_{k+1}}{2}\left[\left(\frac{1}{\eta_k}+
\theta_k L\right)\|x^k-x^{k-1}\|+(\theta_k-\theta_{k+1})\|\nabla g(x^k)\|\right]+\frac{g(x^{k+1})}{2g(x^k)}\e_k\label{ineq11}.
\end{align}

Further, we observe that 
$$\frac{\eta_{k+1}}{2}\left(\frac{1}{\eta_k}+\theta_k L\right)\leq \frac{\eta_{k}}{2}\left(\frac{1}{\eta_k}+\theta_k L\right)
=\frac{1}{2}+\frac{\eta_k\theta_k L}{2}\leq\frac{1}{2}+\frac{\eta_1\theta_1 L}{2} \ \forall k\geq 1.$$
Moreover, $(\nabla g(x^k))_{k\geq 0}$ is bounded and $\lim\sup_{k\rightarrow+\infty}\frac{g(x^{k+1})}{g(x^k)}<+\infty$, due to $g(x^{k+1})\leq M$ ($M>0$) and 
$\liminf_{k\rightarrow+\infty}g(x^k)>0$, which follows from the continuity of $g$, the fact that $(x^k)_{k\geq 0}$ is bounded and Theorem 
\ref{conv-th2}(ii). Thus there exist some positive constants $C_1,C_2>0$ and $k'\geq 0$ such that 
$$\|x^{k+1}-x^k\|\leq\left( \frac{1}{2}+\frac{\eta_1\theta_1 L}{2}\right)\|x^k-x^{k-1}\|+C_1(\theta_k-\theta_{k+1})+C_2\e_k \ \forall k\geq k'.$$

The conclusion follows from Lemma \ref{fejer2} by noticing that $(\theta_k)_{k\geq 1}$ and $\varphi$ are bounded from below. 

(b) It follows from (a) that $(x^k)_{k\geq 0}$ is a Cauchy sequence, hence it is convergent. The conclusion follows from Theorem \ref{conv-th2}. 
\end{proof}

\section{Future work}\label{sec3}

We point out some open questions to be followed in the future related to the solving of the fractional programming problem under investigation:
\begin{itemize}
\item[1.] Is it possible to evaluate in each iteration the functions $f$ and $\delta_S$ separately, which would actually mean that the set $S$ is addressed in the algorithm by means of its projection operator?

\item[2.] How to incorporate in Algorithm \ref{alg-fr-pr-fb}  some extrapolation terms in the sense of Nesterov in order to improve its speed of convergence?

\item[3.] Can one consider also other situations, for instance when $f$ is smooth and $g$ is 
nonsmooth, or even the more general case where both functions are nonsmooth?
\end{itemize}


\begin{thebibliography}{99}

\bibitem{attouch-bolte2009} H. Attouch, J. Bolte, {\it On the convergence of the proximal algorithm for nonsmooth functions involving analytic
features}, Mathematical Programming 116(1-2) Series B, 5--16, 2009

\bibitem{att-b-red-soub2010} H. Attouch, J. Bolte, P. Redont, A. Soubeyran, {\it Proximal alternating minimization and projection
methods for nonconvex problems: an approach based on the Kurdyka-\L{}ojasiewicz inequality}, Mathematics of Operations Research 
35(2), 438--457, 2010

\bibitem{att-b-sv2013} H. Attouch, J. Bolte, B.F. Svaiter, {\it Convergence of descent methods for semi-algebraic and tame problems: 
proximal algorithms, forward-backward splitting, and regularized Gauss-Seidel methods}, Mathematical Programming 137(1-2) Series A, 91--129, 2013

\bibitem{bauschke-book} H.H. Bauschke, P.L. Combettes, {\it Convex Analysis and Monotone Operator Theory in Hilbert Spaces}, 
CMS Books in Mathematics, Springer, New York, 2011

\bibitem{b-d-l2006} J. Bolte, A. Daniilidis, A. Lewis, {\it The \L{}ojasiewicz inequality for nonsmooth subanalytic functions with applications 
to subgradient dynamical systems}, SIAM Journal on Optimization 17(4), 1205--1223, 2006

\bibitem{b-d-l-s2007} J. Bolte, A. Daniilidis, A. Lewis, M. Shota, {\it Clarke subgradients of stratifiable functions}, 
SIAM Journal on Optimization 18(2), 556--572, 2007

\bibitem{b-d-l-m2010} J. Bolte, A. Daniilidis, O. Ley, L. Mazet, {\it Characterizations of \L{}ojasiewicz inequalities:
subgradient flows, talweg, convexity}, Transactions of the American Mathematical Society 362(6), 3319--3363, 2010

\bibitem{b-sab-teb} J. Bolte, S. Sabach, M. Teboulle, {\it Proximal alternating linearized minimization 
for nonconvex and nonsmooth problems}, Mathematical Programming Series A (146)(1--2), 459--494, 2014

\bibitem{bo-van} J.M. Borwein, J.D. Vanderwerff, {\it Convex Functions: Constructions, Characterizations
and Counterexamples}, Cambridge University Press, 2010

\bibitem{bcl} R.I. Bo\c t, E.R. Csetnek, S. L\'aszl\'o, {\it An inertial forward-backward algorithm for the minimization of the sum of 
two nonconvex functions}, EURO Journal on Computational Optimization, DOI 10.1007/s13675-015-0045-8 

\bibitem{c-pesquet-r} E. Chouzenoux, J.-C. Pesquet, A. Repetti, {\it Variable metric forward-backward algorithm for minimizing the sum of a 
differentiable function and a convex function}, Journal of Optimization Theory and its Applications 162(1), 107--132, 2014

\bibitem{clarke} F. Clarke, {\it Functional analysis, Calculus of Variations and Optimal Control}, 
Graduate Texts in Mathematics 264, Springer, London, 2013

\bibitem{cr} J.-P. Crouzeix, J.A. Ferland, S. Schaible, {\it An algorithm for generalized fractional programs}, 
Journal of Optimization Theory and Applications 47(1), 35--49, 1985

\bibitem{dinkelbach} W. Dinkelbach, {\it On nonlinear fractional programming}, Management Science 13, 492--498, 1967 

\bibitem{EkTem} I. Ekeland, R. Temam, {\it Convex Analysis and Variational Problems}, North-Holland Publishing
Company, Amsterdam, 1976

\bibitem{ibaraki1981} T. Ibaraki, {\it Solving mathematical programming problems with fractional objective functions}, in 
S. Schaible and W.T. Ziemba (ed.): Generalized Concavity in Optimization and Eeconomics, 
Academic Press, New York-London, 441--472, 1981

\bibitem{ibaraki1983} T. Ibaraki, {\it Parametric approaches to fractional programs}, Mathematical Programming 26(3), 345--362, 1983

\bibitem{kurdyka1998} K. Kurdyka, {\it On gradients of functions definable in o-minimal structures}, 
Annales de l'institut Fourier (Grenoble) 48(3), 769--783, 1998

\bibitem{lojasiewicz1963} S. \L{}ojasiewicz, {\it Une propri\'{e}t\'{e} topologique des sous-ensembles analytiques r\'{e}els}, 
Les \'{E}quations aux D\'{e}riv\'{e}es Partielles, \'{E}ditions du Centre National de la Recherche Scientifique Paris, 87--89, 1963

\bibitem{boris-carte} B. Mordukhovich, {\it Variational Analysis and Generalized Differentiation, I: Basic Theory, II: Applications}, 
Springer-Verlag, Berlin, 2006

\bibitem{nes} Y. Nesterov, {\it Introductory Lectures on Convex Optimization: A Basic Course}, Kluwer Academic Publishers, Dordrecht, 2004

\bibitem{rock-wets} R.T. Rockafellar, R.J.-B. Wets, {\it Variational Analysis}, Fundamental Principles of Mathematical Sciences 317, 
Springer-Verlag, Berlin, 1998

\bibitem{schaible} S. Schaible, {\it Fractional programming II. On Dinkelbach's algorithm}, Management Science 22(8), 868--873, (1975/76)

\bibitem{Zal-carte} C. Z\u alinescu, {\it Convex Analysis in General Vector Spaces}, World Scientific, Singapore, 2002

\end{thebibliography}
\end{document}